\documentclass[11pts]{amsart}

\usepackage{amsmath,amsthm,amscd,amssymb,bbm,color}
\usepackage{latexsym}
\usepackage{fullpage}
\usepackage{tikz-cd}
\usepackage{mathtools}
\usepackage[toc]{appendix}

\usepackage{dsfont}

\usepackage{setspace}
\usepackage{nohyperref}

\usepackage [english]{babel}
\usepackage [autostyle, english = american]{csquotes}
\MakeOuterQuote{"}

\numberwithin{equation}{section}

\newtheorem{lem}{Lemma}[section]
\newtheorem{prop}{Proposition}[section]
\newtheorem{thm}{Theorem}[section]

\newtheorem{cor}{Corollary}[section]

\newtheorem{rem}{Remark}[section]

\usepackage{thmtools,thm-restate}

\DeclareMathOperator{\R}{\mathbb{R}}

\DeclareMathOperator{\N}{\mathbb{N}}

\DeclareMathOperator{\one}{\mathds{1}}

\newcommand{\dd}{\mathrm{d}}

\title{On the Liouville Function in Short Intervals}
\author{Jake Chinis\vspace{-1.5\baselineskip}}

\begin{document}

\maketitle

\begin{abstract}
    Let $\lambda$ denote the Liouville function. Assuming the Riemann Hypothesis, we prove that
        \begin{align*}
            \int_X^{2X}\Big|\sum_{x\leq n \leq x+h}\lambda(n) \Big|^2\dd x \ll Xh(\log X)^6,
        \end{align*}
    as $X\rightarrow \infty$, provided $h=h(X)\leq \exp\left(\sqrt{\left(\frac{1}{2}-o(1)\right)\log X \log\log X}\right).$ The proof uses a simple variation of the methods developed by Matom{\"a}ki and Radziwi{\l}{\l} in their work on multiplicative functions in short intervals, as well as some standard results concerning smooth numbers.
\end{abstract}

\section{Introduction}

Let $\lambda$ denote the Liouville function; that is, the completely multiplicative function defined by $\lambda(p):=-1$, for all primes $p$. It is well-known that the Prime Number Theorem (PNT) is equivalent to the fact that $\lambda$ exhibits some cancellation in its partial sums; more precisely, that
\begin{align*}
    \sum_{n\leq x}\lambda(n)=o(x),
\end{align*}
as $x\rightarrow \infty.$

The \textit{M{\"o}bius Randomness Law} (see page 338 of \cite{IwaKow}) tells us that $\{\lambda(n)\}_n$ should behave like a sequence of independent random variables, taking on the values $\pm 1$ with equal probability. As a result, we expect "square-root cancellation" in the partial sums for the Liouville function; that is, for any $\epsilon>0$,
\begin{align*}
    \sum_{n\leq x}\lambda(n) \ll x^{\frac{1}{2}+\epsilon},
\end{align*}
as $x\rightarrow \infty.$ In fact, the above estimate is equivalent to the Riemann Hypothesis (RH); see Theorem 14.25(C) of \cite{Titch}. Furthermore, it is possible to quantity $\epsilon$ in terms of $x$; see \cite{Sound}. 

The Liouville function also exhibits cancellation in short intervals, provided that the length of the interval is sufficiently large. Motohashi \cite{motohashi1976} and Ramachandra \cite{Ramachandra1976}, independently, proved that
\begin{align*}
    \sum_{x\leq n \leq x+h}\lambda(n)=o(h),
\end{align*}
provided $h>x^{\frac{7}{12}+\epsilon}$. Assuming RH, this can be improved to $h>x^\frac{1}{2}(\log x)^A$, for some suitable constant $A>0$; see \cite{MontMaier}. Recently, Matom{\"a}ki and Ter{\"a}v{\"a}inen \cite{MatTer_Mob} improved Motohashi and Ramachandra's results to include the larger range $h>x^{\frac{11}{20}+\epsilon}.$

By relaxing the condition that our estimates hold for all short intervals, we can get results that hold for smaller values of $h$. In unpublished work of Gao\footnote{Strictly speaking, Gao's result, and those stated above his, were initially proven for the M{\"o}bius function, $\mu$, but the proofs extend to the Liouville function with little effort.}, it is shown that
\begin{align*}
    \int_{X}^{2X}\Big| \sum_{x\leq n \leq x+h}\lambda(n) \Big|^2\dd x 
    \ll
    Xh(\log X)^A,
\end{align*}
assuming RH, for some large constant $A>0$. 

The preceding results should be compared with what is known for primes in short intervals. More precisely, an equivalent form of the PNT is given by 
\begin{align*}
    \sum_{n\leq x} \Lambda(n) = x + o(x),
\end{align*}
where $\Lambda$ denotes the von Mangoldt function (defined to be $\log p$ if $n$ is a power of the prime $p$ and $0$ otherwise). Furthermore, the Riemann Hypothesis is equivalent to the following estimate:
\begin{align*}
    \sum_{n\leq x} \Lambda(n) = x + O(x^\frac{1}{2}(\log x)^2);
\end{align*}
in particular, RH implies that
\begin{align*}
    \sum_{x\leq n \leq x+h}\Lambda(n) = h+o(h),
\end{align*}
provided $h>x^\frac{1}{2}(\log x)^{2+\epsilon}.$ 

Again restricting ourselves to results that hold almost everywhere, Selberg \cite{Selberg} proved that if RH holds, then
\begin{align*}
    \int_{X}^{2X}\Big| \sum_{x\leq n \leq x+h} \Lambda(n) - h \Big|^2 \dd x 
    \ll
    Xh(\log X)^2,
\end{align*}
so that almost all short intervals contain the correct number of primes. 

It is important to note that both Gao and Selberg obtain square-root cancellation in their estimates. In each case, they relate the short sum to a contour integral of some Dirichlet series and then shift this contour to the edge of the critical strip, picking up any poles along the way. For Selberg, the corresponding Dirichlet series is $\zeta^\prime(s)/\zeta(s)$ and, assuming RH, the only pole will be at $s=1$. Near the half-line, we still need to consider the contributions from the non-trivial zeros, $\rho=1/2+i\gamma$, of $\zeta(s)$, as $\zeta^\prime(s)/\zeta(s)$ is large for $s\approx\rho$. Fortunately, the behaviour of $\zeta^\prime(s)/\zeta(s)$ for $s\approx\rho$ is well understood (it is even conjectured that the zeros of $\zeta(s)$ are simple). For Gao, the corresponding Dirichlet series is $\zeta(2s)/\zeta(s)$, which also has poles at $s=\rho$, but the residues in this case are equal to $\zeta(2\rho)/\zeta^\prime(\rho)$. Since very little is known about $1/\zeta^\prime(\rho)$, we need to proceed in an indirect manner. The key idea is to use a sum over primes to approximate $\zeta(s)$ and then avoid "clusters" of zeros of $\zeta(s)$ on the half-line: near these regions, $1/\zeta(s)$ is large and the contour is chosen so that $1/\zeta(s)$ is not too large, at least in some sense; see \cite{MontMaier}, to get an idea of Gao's proof.

We now turn our attention to the breakthrough work of Matom{\"a}ki and Radziwi{\l}{\l}, where they relate the average value of $1$-bounded multiplicative functions in short intervals to the corresponding average value in large intervals:
\begin{thm}[\cite{MatRad_Mult}] 
\label{MR_thm}
Let $f:\N\rightarrow [-1,1]$ be a multiplicative function and let $h=h(X)\rightarrow \infty$ arbitrarily slowly as $X\rightarrow \infty$. Then, for almost all $x\in[X,2X]$,
\begin{align*}
    \frac{1}{h}\sum_{x\leq n\leq x+h}f(n) = \frac{1}{X}\sum_{X\leq n \leq 2X}f(n) + o(1),
\end{align*}
with $o(1)$ not depending on $f$.
\end{thm}

In the case of the Liouville function, Theorem \ref{MR_thm} implies that
\begin{align*}
    \sum_{x\leq n \leq x+h}\lambda(n)=o(h),
\end{align*}
for almost all $x\in[X,2X]$; in particular,
\begin{align*}
    \int_X^{2X}\Big|\sum_{x\leq n \leq x+h} \lambda(n) \Big|^2\dd x = o(Xh^2),
\end{align*}
for any $h=h(X)\rightarrow \infty$ as $X\rightarrow \infty.$ Although Matom{\"a}ki and Radziwi{\l}{\l} do not get square-root cancellation, their results are unconditional, hold for all $h$ going to infinity, and avoid the complex analytic approach used by both Gao and Selberg; instead, they employ a clever decomposition of the corresponding Dirichlet polynomials, which is done by restricting to integers which have prime factors from certain convenient ranges. For a detailed account of the work in \cite{MatRad_Mult} restricted to the Liouville function, see \cite{Sound_MR}. 

In this paper, we apply a simple variation of the methods developed in \cite{MatRad_Mult} in order to prove the following:

\begin{thm}
\label{main_thm}
Assume the Riemann Hypothesis. Then,
\begin{align*}
        \int_X^{2X}\Big|\sum_{x\leq n\leq x+h}\lambda(n) \Big|^2\dd x 
        \ll
        Xh(\log X)^{6},
\end{align*}
provided $h=h(X)\leq \exp\left(\sqrt{\left(\frac{1}{2}-o(1)\right)\log X \log\log X}\right)$.
\end{thm}

Note that Theorem \ref{main_thm} shows square-root cancellation for the Liouville function in almost all short intervals, provided $(\log X)^{6+\epsilon}\leq h\leq \exp\left(\sqrt{\left(\frac{1}{2}-o(1)\right)\log X \log\log X}\right)$. Of course, Theorem \ref{main_thm} gives an upper bound on the exceptional set of $x\in[X,2X]$ for which square-root cancellation does not hold (via Chebyshev's Inequality). By splitting the short sum into shorter intervals, Theorem \ref{main_thm} can also be used to obtain some cancellation for larger $h$; namely,

\begin{cor}
Assume RH and suppose $h\leq X.$ Then,
\begin{align*}
        \int_X^{2X}\Big|\sum_{x\leq n\leq x+h}\lambda(n) \Big|^2\dd x 
        \ll
        Xh^2(\log X)^{6}\left(\frac{1}{h} + \frac{1}{H}\right),
\end{align*}
where $H=H(X):=\exp\left(\sqrt{\left(\frac{1}{2}-o(1)\right)\log X \log\log X}\right)$.
\end{cor}
\section{Preliminaries}
\label{Preliminaries}

We present here a collection of standard results, which we use freely throughout our paper. The first result is a quantitative version of Perron's Formula, which serves as an approximation to the indicator function on $(1,\infty)$. The second result relates the mean-square of a Dirichlet polynomial to its sum of squares and this will be our main tool in proving Theorem \ref{main_thm}.

\begin{lem}[Quantitative version of Perron's Formula]
\label{EffPer}
Fix $\kappa>0$. Then, 
\begin{align*}
    \frac{1}{2\pi i}\int_{\kappa-iT}^{\kappa+iT}\frac{y^s}{s}\dd s
    =
    \left
        \{\def\arraystretch{1.2}%
            \begin{array}{@{}l@{\quad}l@{}}
                1 & \mbox{if $y>1$}\\
                \frac{1}{2} & \mbox{if $y=1$}\\
                0 & \mbox{if $y<1$}
            \end{array}
    \right.
    +
    \mathcal{O}\Big(\frac{y^\kappa}{\max\{1,T|\log y|\}}\Big),
\end{align*}
uniformly for both $y>0$ and $T>0.$
\end{lem}
\begin{proof}
See \cite[Theorem II.2.3]{Ten}.
\end{proof}

\begin{lem}[Mean Value Theorem]
\label{MVT}
For any sequence of complex numbers $\{a_n\}_n$, we have that
\begin{align*}
    \int_{0}^T\Big|
    \sum_{1\leq n \leq N} a_n n^{it}
    \Big|^2\dd t 
    =
    (T+\mathcal{O}(N))
    \sum_{1\leq n \leq N} |a_n|^2.
\end{align*}
\end{lem}
\begin{proof}
See \cite[Theorem 9.1]{IwaKow}
\end{proof}

\begin{rem}
\label{MVT_rem}
Notice that the main term in Lemma \ref{MVT} corresponds to the contribution from the diagonal terms (as seen by expanding the square and integrating). Notice further that the Mean Value Theorem (MVT) is exceptionally powerful when $N=o(T)$: in this case, the integral is bounded above by the contribution from the diagonal terms and this is best possible.

In their work on multiplicative functions in short intervals, Matom{\"a}ki and Radziwi{\l}{\l} use the MVT in the range $T=o(N)$, which, together with the decomposition of their Dirichlet polynomials, provides a small saving in most cases; see Section 2.1 of \cite{MatRad_Mult}. In this paper, we are interested in the opposite range and decompose our Dirichlet polynomials into two sums, one of which has length $N=o(T)$ and another which can be handled trivially (at least for small $h$). 
\end{rem}

We need two more preliminary results. The first is a pointwise bound on $\sum_{X \leq n \leq 2X}\lambda(n)/n^{\frac{1}{2}+it},$ which allows us to remove the contribution from the small values of $t$ in our average value (this can be thought of as the analogous result to Lemma 1 in \cite{MatRad_Liou}). The second is an analogue of Lemma 2 in \cite{MatRad_Mult}: we need a pointwise bound on sums of the form $\sum_{P\leq p \leq 2P}1/p^{\frac{1}{2}+\frac{1}{\log X}+it},$ for large values of $t$.
\begin{lem}
\label{Small_t}
Assume RH. Then,
\begin{align*}
    \sum_{X \leq n \leq 2X}\frac{\lambda(n)}{n^{\frac{1}{2}+it}}\ll_\epsilon (1+|t|)^\epsilon X^\epsilon,
\end{align*}
for all $\epsilon>0$ and all $t\in \R$.
\end{lem}
\begin{proof}
By Lemma 3.12 of \cite{Titch}, we have that
\begin{align*}
    \sum_{X \leq n \leq 2X}\frac{\lambda(n)}{n^{\frac{1}{2}+it}}
    &=
    \frac{1}{2\pi i}\int_{2-iT}^{2+iT}\frac{\zeta(2s+1+2it)}{\zeta(s+\frac{1}{2}+it)}\frac{(2X)^s-X^s}{s}\dd s 
    +
    \mathcal{O}\Big(\frac{X^2}{T}\Big).
\end{align*}
Given $\epsilon>0$ and assuming RH, the function $\zeta(2(s+\frac{1}{2}+it))/\zeta(s+\frac{1}{2}+it)$ is analytic for $\Re(s)\geq \epsilon.$ Shifting the contour to the edge of this region, we get that
\begin{align*}
    \sum_{X \leq n \leq 2X}\frac{\lambda(n)}{n^{\frac{1}{2}+it}}
    &=
    \frac{-1}{2\pi i}\Biggr(
    \int_{2+iT}^{\epsilon+iT}
    + 
    \int_{\epsilon+iT}^{\epsilon-iT}
    +
    \int_{\epsilon-iT}^{2-iT}\Biggr)
    \Big(\frac{\zeta(2s+1+2it)}{\zeta(s+\frac{1}{2}+it)}\frac{(2X)^s-X^s}{s}\dd s \Big)
    +
    \mathcal{O}\Biggr(\frac{X^2}{T} \Biggr).
\end{align*}
Then, using the facts that $\zeta(s)\ll 1/(\Re(s)-1)$, for $\Re(s)>1$, and $1/\zeta(s)\ll |\Im(s)|^\epsilon$, for $\Re(s)\geq \frac{1}{2}+\epsilon,$ (see 14.2.6 in \cite{Titch}), we have that
\begin{align*}
    \sum_{X \leq n \leq 2X}\frac{\lambda(n)}{n^{\frac{1}{2}+it}}
    &\ll_\epsilon
    (|t|+T)^\epsilon X^\epsilon \int_{-T}^T  \frac{\dd y}{\sqrt{\epsilon^2+y^2}} + (|t|+T)^\epsilon\frac{X^2}{T}\\
    &\ll_\epsilon
    (|t|+T)^\epsilon(X^\epsilon \log T + \frac{X^2}{T}).
    \end{align*}
    Taking $T=X^2$, this boils down to
    \begin{align*}
        \sum_{X \leq n \leq 2X}\frac{\lambda(n)}{n^{\frac{1}{2}+it}}
        \ll_\epsilon
        (|t|+X^2)^\epsilon X^\epsilon
        \ll_\epsilon
        (1+|t|)^\epsilon X^\epsilon,
    \end{align*}
as claimed.
\end{proof}

\begin{lem}
\label{Large_t}
Assume RH and suppose $P\leq X$. Then,  
\begin{align*}
    \sum_{P\leq p \leq 2P}\frac{1}{p^{\frac{1}{2}+\frac{1}{\log X}+it}}
    \ll \log X,
\end{align*}
uniformly for $X^\frac{1}{2}\leq |t|\leq X.$
\end{lem}
\begin{proof}
Note that 
\begin{align}
\label{int-alpha}
    \sum_{P\leq p \leq 2P}\frac{1}{p^{\frac{1}{2}+\frac{1}{\log X}+it}}
&=
    \int_{1/\log X}^{1/2}
    \sum_{P\leq n \leq 2P}\frac{\Lambda(n)}{n^{\frac{1}{2}+\alpha+it}}
    \dd \alpha
    +
    \mathcal{O}\left(1\right).
 \end{align}
Applying Perron's Formula (Lemma \ref{EffPer}) to the integrand, with $\kappa := {1}/{2}-\alpha +1/\log X>0$ and $T:=P^{\frac{1}{2}}/2$, we have that
\begin{align*}
    \sum_{P\leq n \leq 2P}\frac{\Lambda(n)}{n^{\frac{1}{2}+\alpha+it}}
&=
    \frac{1}{2 \pi i}\int_{{1}/{2}-\alpha + 1/\log X - iT}^{{1}/{2} - \alpha + 1/\log X +iT} -\frac{\zeta^\prime}{\zeta}(s + \frac{1}{2} + \alpha + it)\frac{(2P)^s-P^s}{s}\dd s
    +
    \mathcal{O}\left(P^{-\alpha}\log X\right).
 \end{align*}
 A standard contour shift argument, using Chapters 12 and 13 of \cite{Mont-Vaughan}, shows that this last expression is $\ll P^{-\alpha}(\log P)\log X$. Upon inserting this bound back into (\ref{int-alpha}), and integrating over $\alpha$, we get the desired result. Below, we provide some details for convenience to the reader.
 
Assuming RH and shifting the contour to the line $\Re(s) = -3/2-\alpha=:\sigma_0$, we have that\footnote{Strictly speaking, we need to choose $T$ in such a way that the horizontal line segments in the contour (shifted by $t$) are bounded away from ordinates $\gamma$ of $\zeta$; this can be done via Lemma 12.2 of \cite{Mont-Vaughan}.}
\begin{align*}
        \sum_{P\leq n \leq 2P}\frac{\Lambda(n)}{n^{\frac{1}{2}+\alpha+it}}
        &=
        \frac{-1}{2\pi i}\left(
        \int_{\kappa+iT}^{\sigma_0+iT}
        + 
        \int_{\sigma_0+iT}^{\sigma_0-iT}
        +
        \int_{\sigma_0-iT}^{\kappa-iT}\right)
        \left(- \frac{\zeta^\prime}{\zeta}(s+\frac{1}{2}+\alpha+it)\frac{(2P)^s-P^s}{s}\dd s \right)\\
        &\qquad\qquad+
        \sum_{\substack{{\rho}\\{|\gamma-t|\leq T}}}\frac{(2P)^{-\alpha-it+i\gamma}-P^{-\alpha -it +i\gamma}}{-\alpha - it +i\gamma} + \mathcal{O}\left(P^{-\alpha}\log X\right),
\end{align*}
where the sum is over the non-trivial zeros, $\rho={1}/{2}+i\gamma$, of $\zeta$ (counted with multiplicity). Notice that the sum over $\rho$ corresponds to the residues of the integrand at $s=\rho-{1}/{2}-\alpha-it$. Notice further that we do not pick up the pole of $\zeta$ at $1$, as $\Im(s)+t$ is bounded away from $0$ (recall that $X^\frac{1}{2}\leq|t|\leq X$ with $T=P^\frac{1}{2}/2$ and $P\leq X$).

To deal with the vertical integral, we use the fact that ${\zeta^\prime(s)}/{\zeta}(s)\ll \log (|s|+1)$ uniformly for $\Re(s)\leq -1$, provided $s$ is not close to an even integer (i.e., provided we are not close to a trivial zero of $\zeta$); see Lemma 12.4 of \cite{Mont-Vaughan}. Using the above and bounding the vertical integral trivially produces an admissible error.

For the horizontal integrals, write $s=\sigma \pm iT$ and note that
\begin{align*}
    \left|\int_{\sigma_0 \pm iT}^{\kappa \pm iT}\frac{\zeta^\prime}{\zeta}(\sigma + \frac{1}{2}+\alpha +i(t\pm T))\frac{(2P)^{\sigma\pm iT}-P^{\sigma \pm iT}}{\sigma \pm iT}\dd \sigma\right|
    &\ll
    \frac{P^{\frac{1}{2}-\alpha}}{T}\int_{-1}^{1+1/\log X}\left|\frac{\zeta^\prime}{\zeta}(\sigma + i(t\pm T))\right|\dd \sigma\\
    &\ll 
    P^{-\alpha}\int_{-1}^{1+1/\log X}\left|\frac{\zeta^\prime}{\zeta}(\sigma + i(t\pm T))\right|\dd \sigma
\end{align*}
which follows by making the change of variables $\sigma \mapsto \sigma - {1}/{2}-\alpha$, taking a pointwise bound on the integrand, and recalling that $\kappa = {1}/{2}-\alpha-it$, $\sigma_0=-3/2-\alpha$, and $T=P^{\frac{1}{2}}/2$. From the functional equation, it suffices to consider the integral over $\sigma\in[{1}/{2},1+1/\log X]$ alone, which we break into three intervals, according to the bounds one can get for ${\zeta^\prime}/{\zeta}$:
\begin{enumerate}
    \item By Lemma 12.2 of \cite{Mont-Vaughan}, we know that $\frac{\zeta^\prime}{\zeta}(\sigma + i(t\pm T))\ll (\log X)^2,$ uniformly for $-1\leq \sigma\leq 2$; in particular,
    \begin{align*}
         P^{-\alpha}\int_{{1}/{2}}^{{1}/{2}+{1}/{\log X}}\left|\frac{\zeta^\prime}{\zeta}(\sigma + i(t\pm T))\right|\dd \sigma
         \ll
         P^{-\alpha}\log X.
    \end{align*}
    \item By Lemma 13.20 of \cite{Mont-Vaughan}, we know that 
    \begin{align*}
        \frac{\zeta^\prime}{\zeta}(\sigma +i(t\pm T)) 
        =
        \sum_{\substack{{\rho}\\{|\gamma-(t\pm T)|\leq 1/\log\log ((|t\pm T|+4)}}}\frac{1}{\sigma-1/2 + i(t\pm T -\gamma)}
        +\mathcal{O}(\log X),
    \end{align*}
    uniformly for $|\sigma-1/2|\leq 1/\log\log(|t\pm T|+4)$. Using the trivial bound 
    $$\frac{1}{|{\sigma-1/2 + i(t\pm T -\gamma)}|}\leq \frac{1}{|\sigma-1/2|},$$
    and integrating over $\sigma$, we then have that
    \begin{align*}
        P^{-\alpha}\int_{{1}/{2}+1/\log X}^{{1}/{2}+1/{\log\log (|t\pm T|+4)}}\left|\frac{\zeta^\prime}{\zeta}(\sigma + i(t\pm T))\right|\dd \sigma
       &\ll
        P^{-\alpha}\log X,
    \end{align*}
    which follows from the fact that the integral over $\sigma$ is bounded by $\log\log X$ and from the fact that there are $\ll \log X/\log\log X$ zeros in each window of length $1/\log\log X$ (see Lemma 13.19 of \cite{Mont-Vaughan}).
    \item Finally, for $1/2+1/\log\log (|t\pm T|+4)\leq \sigma \leq 1+1/\log X$, we use the fact that 
    \begin{align*}
        \frac{\zeta^\prime}{\zeta}(\sigma + i(t\pm T))\ll (\log X)^{2-2\sigma}\log\log X,
    \end{align*}
    uniformly for $\sigma$ in the desired range (see Corollary 13.14 of \cite{Mont-Vaughan}), which yields:
    \begin{align*}
            P^{-\alpha}\int_{1/2+1/\log\log (|t\pm T|+4)}^{1+1/\log X}
            \left|\frac{\zeta^\prime}{\zeta}(\sigma + i(t\pm T))\right|\dd \sigma
            &\ll
            P^{-\alpha}(\log\log X) \int_{1/2+1/\log\log(|t\pm T|+4)}^{1+1/\log X}(\log X)^{2-2\sigma}\dd \sigma\\
            &\ll
            P^{-\alpha}\log X.
    \end{align*}
\end{enumerate}

In other words, a standard contour shift yields the following:
\begin{align*}
        \sum_{P\leq n \leq 2P}\frac{\Lambda(n)}{n^{\frac{1}{2}+\alpha+it}}
       &\ll
       P^{-\alpha}\sum_{\substack{{\rho}\\{|\gamma-t|\leq P^{\frac{1}{2}}/2}}}
       \left|\frac{2^{-\alpha-it+i\gamma}-1}{-\alpha - it +i\gamma}\right| 
       +
       P^{-\alpha}\log X\\
       &\ll
       P^{-\alpha}
       \sum_{\substack{{\rho}\\{|\gamma-t|\leq P^{\frac{1}{2}}/2}}}
       \frac{1}{1+|t-\gamma|} 
       +
       P^{-\alpha}\log X\\
       &\ll
       P^{-\alpha}(\log P)\log X,
\end{align*}
where the last line follows from the fact that there are $\ll \log X$ zeros in each interval of length $1$ (counted with multiplicity) and recalling that the sum over $\rho$ comes from the contribution of the residues in our contour integral. Upon substituting the last bound into (\ref{int-alpha}), and integrating over $\alpha$, we obtain the desired result.
\end{proof}

\begin{rem}
The proof of Theorem \ref{main_thm} can easily be adapted to a more general setting. For an arbitrary multiplicative function $f$, all we need is an analogue to Lemma \ref{Large_t}. Essentially, we are looking for square-root cancellation in the corresponding sum over primes:
\begin{align*}
    \sum_{P\leq p \leq 2P}f(p)p^{it}\ll P^\epsilon \Big(\sum_{P\leq p \leq 2P}\big(f(p)\big)^2 \Big)^\frac{1}{2}.
\end{align*}
For example, the above estimate is known to hold, assuming RH, for coefficients of automorphic forms and for multiplicative functions of the form $\mu(n) \lambda_\pi(n)$ or $\lambda_\pi(n)$, where  $\mu$ is the M{\"o}bius function and where the $\lambda_\pi(n)$'s are the coefficients of an automorphic representation $\pi$.
\end{rem}
\section{Initial Reductions}
\label{InitialRed}

In this section, we reduce our problem to that of bounding the mean square of a Dirichlet polynomial. We begin with the following standard lemma, which essentially follows from Perron's Formula (together with a few other tricks):
\begin{lem}[Plancherel]
\label{x-t}
For $1\leq h \leq X,
$\begin{align*}
    \frac{1}{X}\int_X^{2X} 
    \Big|\frac{1}{h} 
    \sum_{x\leq n \leq x+h}\lambda(n)
    \Big|^2 \dd x
    &\ll
    \frac{1}{X}
    \int_0^{X/h}\Big|
    \sum_{X\leq n \leq 4X}\frac{\lambda(n)}{n^{\frac{1}{2}+ it}}
    \Big|^2 \dd t
    +
    \max_{T>X/h}
    \frac{1}{hT}
    \int_T^{2T}\Big|
    \sum_{X\leq n \leq 4X}\frac{\lambda(n)}{n^{\frac{1}{2}+ it}}
    \Big|^2 \dd t.
\end{align*}
\end{lem}
\begin{proof}
See \cite[Lemma 14]{MatRad_Mult} (equivalently, \cite[Lemma 4]{MatRad_Liou}).
\end{proof}

Using Lemma \ref{Small_t}, we can remove the contribution from the small values of $t$ in the RHS of Lemma \ref{x-t}: for any $\epsilon>0$,
\begin{align*}
    \frac{1}{X}
    \int_0^{X^{\frac{1}{2}}} 
    \Big|
    \sum_{X \leq n \leq 4X}\frac{\lambda(n)}{n^{\frac{1}{2}+it}}
    \Big|^2
    \dd t
    \ll_\epsilon
    X^{\epsilon-\frac{1}{2}},
\end{align*}
which follows by bounding the integrand pointwise. Thus, Theorem \ref{main_thm} follows from Lemma \ref{x-t} once we show that
\begin{align}
\label{goal}
    \frac{1}{hT}
    \int_{X^\frac{1}{2}}^T 
    \Big|
    \sum_{X \leq n \leq 4X}\frac{\lambda(n)}{n^{\frac{1}{2}+it}}
    \Big|^2
    \dd t
    \ll
    \frac{(\log X)^6}{h},
\end{align}
for $T\geq X/h$, where we have assumed, w.l.o.g., that $h\leq X^{\frac{1}{2}}$ (this is done to avoid the degenerate case when $T=X/h\leq X^{1/2}$).  

For $T>X$, the Mean Value Theorem (Lemma \ref{MVT}) immediately gives the desired bound:
\begin{align*}
        \frac{1}{hT}
    \int_{X^\frac{1}{2}}^T 
    \Big|
    \sum_{X \leq n \leq 4X}\frac{\lambda(n)}{n^{\frac{1}{2}+it}}
    \Big|^2
    \dd t
    \ll
    \frac{1}{hT}(T+X)\sum_{X \leq n \leq 4X}\frac{1}{n}\ll \frac{1}{h}.
\end{align*} 
It remains to prove that (\ref{goal}) holds for $X/h\leq T\leq X$. Our main tool is still the Mean Value Theorem, but this is only winning if we can shorten the length our sum: we should think of $T\approx X/h$ and recall that the MVT is best possible if the length of the corresponding Dirichlet polynomial is of size $N\approx T$; see Remark \ref{MVT_rem}. Since our Dirichlet polynomial has length $X$, our goal is to split the sum over $X \leq n \leq 4X$ into two sums, one of which has length $\approx X/h$ and another which can be handled separately. This is done by considering the integers $X \leq n \leq 4X$ which have a prime factor $p>h$: by writing such integers as $n=pm$, and then factoring out the sum over $p$, the sum over $m \approx X/p$ that remains has length $\ll X/h$; using the MVT on the sum over $m$ is now winning. In Section \ref{SmallPrimes}, we deal with the remaining integers, all of whose primes factors are $\leq h$. Fortunately for us, there are few of these so-called $h$-smooth integers, at least for small $h$, so that the MVT can be applied directly to give us what we want. Getting square-root cancellation for larger $h$ will require some new ideas; this is ongoing work.
\section{Integers with large prime factors}
\label{LargePrimes}

For the integers $X \leq n \leq 4X$ which have at least one prime factor $p>h$, we have the following:

\begin{prop}
\label{Rough}
Assume RH. If $X/h\leq T \leq X$, then
\begin{align*}
\frac{1}{hT}
\int_{X^\frac{1}{2}}^T\Big|
\sum_{\substack{{X \leq n \leq 4X}\\{\exists p>h:p|n}}}
\frac{\lambda(n)}{n^{\frac{1}{2}+it}}
\Big|^2\dd t
\ll
\frac{(\log X)^{6}}{h}.
\end{align*}
\end{prop}

\begin{proof}
Note that
\begin{align*}
    \sum_{\substack{{X \leq n \leq 4X}\\{\exists p|n:p>h}}}
    \frac{\lambda(n)}{n^{\frac{1}{2}+it}}
    &=
    \sum_{h<p\leq 4X}
    \frac{\lambda(p)}{p^{\frac{1}{2}+it}}
    \sum_{X/p \leq m \leq  4X/p}
    \frac{1}{m^{\frac{1}{2}+it}}
    \frac{\lambda(m)}{\#\{\text{$q$ prime}:q>h,q|m\}+\one_{p\nmid m}},
\end{align*}
where $\#\{\text{$q$ prime}:q>h,q|m\}+\one_{p\nmid m}$ counts\footnote{In the published version of \cite{MatRad_Mult}, the term $\one_{p\nmid m}$ appears as the constant $1$, but this was corrected in later versions of their paper. In any case, we will see that this misprint does not affect their argument.} the number of ways $X \leq n \leq 4X$ can be factored as $n=mp$ with $p>h$. As Matom{\"a}ki and Radziwi{\l}{\l} remark in \cite{MatRad_Mult}, this decomposition is analogous to Buchstab's identity, which is a variant of Ramar{\'e}'s identity; see Section 17.3 of \cite{FriedIwa}.

Our goal now is to remove the dependence on $\one_{p\nmid m},$ which is done by splitting the inner sum into those $m$ for which $p\nmid m$ and $p|m$, respectively:
\begin{align*}
    &
                \sum_{h<p\leq 4X}
                \frac{\lambda(p)}{p^{\frac{1}{2}+it}}
                \sum_{X/p \leq m \leq 4X/p}
                \frac{\lambda(m)}{m^{\frac{1}{2}+it}}
                \frac{1}{\#\{q>h:q|m\}+\one_{p\nmid m}}\\
    &\qquad\qquad
    =
                \sum_{h<p\leq 4X}
                \frac{\lambda(p)}{p^{\frac{1}{2}+it}}
                \sum_{\substack{{X/p \leq m \leq 4X/p}\\{p\nmid m}}}
                \frac{1}{m^{\frac{1}{2}+it}}
                \frac{\lambda(m)}{\#\{q>h:q|m\}+1}
    +
                \sum_{h<p\leq 4X}
                \frac{\lambda(p)}{p^{\frac{1}{2}+it}}
                \sum_{\substack{{X/p \leq m \leq 4X/p}\\{p|m}}}
                \frac{1}{m^{\frac{1}{2}+it}}
                \frac{\lambda(m)}{\#\{q>h:q|m\}},
\end{align*}
with $q$ varying over the set of primes. This can be simplified further by adding and subtracting all $m$ for which $p|m$ to the first term  and setting
\begin{align*}
    a_m:=\frac{-\lambda(m)}{\#\{q>h:q|m\}+1},\; 
    b_m:=\frac{-\lambda(m)}{\#\{q>h:q|m\}(\#\{q>h:q|m\}+1)},
\end{align*}
which yields
\begin{align*}
    \sum_{h<p\leq 4X}\frac{-1}{p^{\frac{1}{2}+it}}
    &\sum_{X/p \leq m \leq 4X/p}\frac{\lambda(m)}{m^{\frac{1}{2}+it}}\frac{1}{\#\{q>h:q|m\}+\one_{p\nmid m}}\\
    &\qquad\qquad=
    \sum_{h<p\leq 4X}
    \frac{1}{p^{\frac{1}{2}+it}}
    \sum_{X/p \leq m \leq 4X/p}\frac{a_m}{m^{\frac{1}{2}+it}}
    +
    \sum_{h<p\leq 4X}
    \frac{1}{p^{\frac{1}{2}+it}}
    \sum_{\substack{{X/p \leq m \leq 4X/p}\\{p|m}}}\frac{b_m}{m^{\frac{1}{2}+it}}\\
    &\qquad\qquad=
    \sum_{h<p\leq 4X}
    \frac{1}{p^{\frac{1}{2}+it}}
    \sum_{X/p \leq m \leq 4X/p}\frac{a_m}{m^{\frac{1}{2}+it}}
    +
    \sum_{h<p\leq 4X}
    \frac{1}{p^{1+2it}}
    \sum_{X/p^2\leq m \leq  4X/p^2}\frac{b_{mp}}{m^{\frac{1}{2}+it}},
\end{align*}
and where the last line follows by writing $m=mp$ in the second double sum. Thus,
\begin{align*}
                \frac{1}{hT}
                \int_{X^\frac{1}{2}}^{T}\Big|
                \sum_{\substack{{X \leq n \leq 4X}\\{\exists p|n:p>h}}}
                \frac{\lambda(n)}{n^{\frac{1}{2}+it}}
                \Big|^2\dd t
     &\ll
                \frac{1}{hT}
                \int_{X^\frac{1}{2}}^{T}\Big|
                \sum_{h<p\leq 4X}\frac{1}{p^{\frac{1}{2}+it}}
                \sum_{X/p \leq m \leq 4X/p}
                \frac{a_m}{m^{\frac{1}{2}+it}}
                \Big|^2\dd t\\
    &
                \qquad\qquad+
                \frac{1}{hT}
                \int_{X^\frac{1}{2}}^{T}\Big|
                \sum_{h<p\leq 4X}
                \frac{1}{p^{1+2it}}
                \sum_{X/p^2\leq m \leq 4 X/p^2}
                \frac{b_{mp}}{m^{\frac{1}{2}+it}}
                \Big|^2\dd t.
\end{align*}

Applying the Mean Value Theorem (Lemma \ref{MVT}) to the second integral, we see that
\begin{align*}
    \frac{1}{hT}
    \int_{X^\frac{1}{2}}^T\Big|
    &\sum_{h<p<4X}\frac{1}{p^{1+2it}}
    \sum_{X/p^2 \leq m \leq 4X/p^2}\frac{b_{mp}}{m^{\frac{1}{2}+it}}\Big|^2\dd t\\
    &\qquad\qquad\ll
    \frac{1}{hT}(T+X)
    \sum_{h<p<4X}\frac{1}{p^2}
    \sum_{X/p^2 \leq m \leq 4X/p^2}\frac{1}{m}
    \ll
    \frac{1}{h},
\end{align*}
recalling that $X/h\leq T \leq X$.

For the remaining integral, we wish to separate the variables $p$ and $m$, so that we may apply a pointwise bound to the sum over $p$. We begin by splitting the sum over $p$ into dyadic intervals: 
\begin{align*}
    \frac{1}{hT}
    \int_{X^\frac{1}{2}}^{T}\Big|
    \sum_{h<p\leq 4X}
    &\frac{1}{p^{\frac{1}{2}+it}}
    \sum_{X/p \leq m \leq 4X/p}
    \frac{a_m}{m^{\frac{1}{2}+it}}
    \Big|^2\dd t\\
    &=
    \frac{1}{hT}
    \int_{X^\frac{1}{2}}^{T}\Big|
    \sum_{j = \lfloor \log_2 h \rfloor}^{\lfloor \log_2 4X \rfloor - 1}
    \sum_{2^j<p\leq 2^{j+1}}
    \frac{1}{p^{\frac{1}{2}+it}}
    \sum_{\substack{
    {X/2^{j+1}\leq m \leq X/2^{j-1}}\\{X/p \leq m \leq 4X/p}}}
    \frac{a_m}{m^{\frac{1}{2}+it}}
    \Big|^2\dd t\\
    &\ll
    \frac{(\log X)^2}{hT}
    \max_{\log_2 h\leq j \leq \log_2(4X)}
    \int_{X^\frac{1}{2}}^T\Big|
    \sum_{2^j<p\leq 2^{j+1}}
    \frac{1}{p^{\frac{1}{2}+it}}
    \sum_{\substack{
    {X/2^{j+1}\leq m \leq X/2^{j-1}}\\{X/p \leq m \leq 4X/p}}}
    \frac{a_m}{m^{\frac{1}{2}+it}}
    \Big|^2\dd t,
\end{align*}
where $\log_2$ denotes the base $2$ logarithm function and where the last line follows by taking the absolute value inside the sum over $j$, noting that there are $\ll \log X/h\leq \log X$ such dyadic intervals.

Using Lemma \ref{EffPer} with $\kappa=1/\log X$, we can remove the condition that $X \leq mp \leq 4X$:
\begin{align*}
    \one_{X \leq mp \leq 4X}
    =
    \frac{1}{2\pi i}
    \int_{\kappa-iY}^{\kappa+iY}
    \frac{(4X)^s-X^s}{(mp)^s}
    \frac{\dd s}{s}
    +
    \mathcal{O}\Big(
    \frac{1/(mp)^\kappa}{\max\{1,Y|\log X/mp|\}}
    \Big),
\end{align*}
which yields:
\begin{align*}
        \frac{1}{hT}
        \int_{X^\frac{1}{2}}^{T}\Big|
        &\sum_{h<p\leq 4X}
        \frac{1}{p^{\frac{1}{2}+it}}
        \sum_{X/p \leq m \leq 4X/p}
        \frac{a_m}{m^{\frac{1}{2}+it}}
        \Big|^2\dd t\\
        &\qquad\qquad\ll
        \frac{(\log X)^2}{hT}
        \int_{X^\frac{1}{2}}^{T}\Big|
        \int_{\kappa-iY}^{\kappa+iY}
        \frac{(4X)^s-X^s}{s}
        \sum_{2^j<p\leq 2^{j+1}}
        \frac{1}{p^{s+\frac{1}{2}+it}}
        \sum_{X/2^{j+1}\leq m \leq X/2^{j-1}}
        \frac{a_m}{m^{s+\frac{1}{2}+it}}
        \dd s
        \Big|^2\dd t\\
        &\qquad\qquad\qquad\qquad+
        \frac{(\log X)^2}{h}
        \Big(
        \frac{X^\frac{1}{2}\log X}{Y}
        \Big)^2,
\end{align*}
for some ${\log_2 h\leq j \leq \log_2(4X)}$

We can now use Minkowski's Inequality for integrals (see Section A.1 of \cite{Stein}) to change the order of integration:
\begin{align*}
    &
        \int_{X^\frac{1}{2}}^{T}\Big|
        \int_{\kappa-iY}^{\kappa+iY}
        \frac{(4X)^s-X^s}{s}
        \sum_{2^j<p\leq 2^{j+1}}
        \frac{1}{p^{s+\frac{1}{2}+it}}
        \sum_{X/2^{j+1}\leq m \leq X/2^{j-1}}
        \frac{a_m}{m^{s+\frac{1}{2}+it}}
        \dd s
        \Big|^2\dd t\\
    &\qquad\qquad\ll
        \Biggr(
        \int_{\kappa-iY}^{\kappa+iY}
        \Biggr(
        \int_{X^\frac{1}{2}}^{T}
        \Big|
        \frac{(4X)^s-X^s}{s}
        \sum_{2^j<p\leq 2^{j+1}}
        \frac{1}{p^{s+\frac{1}{2}+it}}
        \sum_{X/2^{j+1}\leq m \leq X/2^{j-1}}
        \frac{a_m}{m^{s+\frac{1}{2}+it}}
        \Big|^2\dd t
        \Biggr)^\frac{1}{2}
        \dd s        
        \Biggr)^2.
\end{align*}
Then, by taking $Y=X^\frac{1}{2}/2$, we can apply Lemma \ref{Large_t} to bound the sum over $p$; this yields the upper bound
\begin{align*}
        &\ll
        (\log X)^2
        \Biggr(
        \int_{\kappa-iY}^{\kappa+iY}
        \frac{1}{|s|}
        \Biggr(
        \int_{X^\frac{1}{2}}^{T}
        \Big|
        \sum_{X/2^{j+1}\leq m \leq X/2^{j-1}}
        \frac{a_m}{m^{s+\frac{1}{2}+it}}
        \Big|^2\dd t
        \Biggr)^\frac{1}{2}
        \dd s        
        \Biggr)^2\\
        &\ll
        (\log X)^4
        (T+X/h)
        \sum_{X/2^{j+1}\leq m \leq X/2^{j-1}}
        \frac{1}{m}\\
        &\ll
        (\log X)^4
        (T+X/h),
\end{align*}
where the second to last line follows from the Mean Value Theorem, recalling that $j\geq \log h/\log 2$, and where the additional powers of $\log X$ come from the bound
\begin{align*}
        \int_{\kappa-iY}^{\kappa+iY}
        \frac{1}{|s|}
        \dd s
        \ll
        \log X,
\end{align*}
recalling that $\kappa=1/\log X$ and $Y=X^\frac{1}{2}/2.$

Therefore, 
\begin{align*}
    \frac{1}{hT}
    \int_{X^\frac{1}{2}}^T\Big|
    \sum_{\substack{{X \leq n \leq 4X}\\{\exists p>h:p|n}}}
    \frac{\lambda(n)}{n^{\frac{1}{2}+it}}
    \Big|^2\dd t
    &\ll
    \frac{(\log X)^{6}}{h},
\end{align*}
provided that $X/h\leq T \leq X$, which is the desired result.
\end{proof}

\begin{rem}
In the proof of Proposition \ref{Rough}, we used a contour integral to separate the variables $p$ and $m$, which lost two factors of $\log X$. As the referee remarked, introducing a smoothing at the beginning of our argument makes the separation of variables loss-less: after Lemma \ref{x-t} (Plancherel), we are working with smooth sums and we can use the rapid decay of the Mellin Transform to ensure that the integral of $1/|s|$, which produced the original factors of $\log X$, converges. This would yield the upper bound $\ll Xh(\log X)^4$ in Theorem \ref{main_thm}. If we truly wish to be pedantic, our current methods yield the upper bound $\ll Xh((\log X) (\log X/h))^2,$ where the $\log X$ factors come from the bound on the sum over primes (Lemma \ref{Large_t}) and where the $\log X/h$ factors come from splitting the sum over primes $h<p\leq 4X$ into dyadic intervals.
\end{rem}

To complete the proof of Theorem \ref{main_thm}, it remains to consider the $h$-smooth integers. The next section is dedicated to this task.
\section{Smooth integers}
\label{SmallPrimes}

For the integers $X \leq n \leq 4X$ all of whose prime factors are $\leq h$, our goal is to obtain the following estimate:
\begin{align*}
        \frac{1}{hT}
        \int_{X^\frac{1}{2}}^T\Big|    
        \sum_{\substack{{X \leq n \leq 4X}\\{p|n \Rightarrow p\leq h}}}
        \frac{\lambda(n)}{n^{\frac{1}{2}+it}}
        \Big|^2\dd t
        \ll
        \frac{1}{h},
\end{align*}
for $X/h\leq T \leq X$ and $h$ as large as possible. This will be accomplished using the MVT (Lemma \ref{MVT}), together with some standard results concerning smooth numbers.

To begin, fix $\epsilon>0$ and let $\Psi(x,y)$ denote the number of $y$-smooth integers up to $x$. By writing $x=y^u$, we have that 
\[\Psi(x,y)=xu^{-(1+o(1))},\]
uniformly in the range $u\leq y^{1-\epsilon}$, as both $y$ and $u$ tend to infinity; see Corollary 1.3 \cite{HildTen}, for example. In particular, 
\[
\Psi(x,y)\ll \frac{x}{y},
\]
provided $y\leq\exp\left(\sqrt{\left(\frac{1}{2}-o(1)\right)\log x \log\log x}\right)$. Using the above, together with the MVT, we then have that
\begin{align*}
        \frac{1}{hT}
        \int_{X^\frac{1}{2}}^T\Big|    
        \sum_{\substack{{X \leq n \leq 4X}\\{p|n \Rightarrow p\leq h}}}
        \frac{\lambda(n)}{n^{\frac{1}{2}+it}}
        \Big|^2\dd t
        &\ll
        \frac{1}{hT}(T+X)
        \sum_{\substack{{X \leq n \leq 4X}\\{p|n \Rightarrow p\leq h}}}
        \frac{1}{n}\\
        &\ll
       \frac{1}{hT}(T+X)\frac{\Psi(4X,h)}{X}\\
       &\ll
       \frac{1}{h},
\end{align*}
for $h\leq \exp\left(\sqrt{\left(\frac{1}{2}-o(1)\right)\log X \log\log X}\right)$, recalling that $X/h\leq T \leq X$. In other words,

\begin{prop}
\label{Smooth}
Suppose $h=h(X)\rightarrow \infty$ as $X\rightarrow \infty$, then
    \begin{align*}
        \frac{1}{hT}
        \int_{X^\frac{1}{2}}^T\Big|    
        \sum_{\substack{{X \leq n \leq 4X}\\{p|n \Rightarrow p\leq h}}}
        \frac{\lambda(n)}{n^{\frac{1}{2}+it}}
        \Big|^2\dd t
        \ll
        \frac{1}{h},
    \end{align*}
provided $h\leq \exp\left(\sqrt{\left(\frac{1}{2}-o(1)\right)\log X \log\log X}\right)$.
\end{prop}

Combining Propositions \ref{Rough} and \ref{Smooth} yields (\ref{goal}), which, together with the initial reductions from Section \ref{InitialRed}, yields Theorem \ref{main_thm}.

\section{Acknowledgements}
The author would like to thank Maksym Radziwi{\l}{\l} for suggesting this problem and for his continuous support throughout the research phase of this paper, Andrew Granville for simplifying the argument originally presented in Section \ref{SmallPrimes}, and the anonymous referee for carefully reviewing this manuscript and whose comments have not only improved the overall exposition of this paper, but which have also strengthened the main theorem. In addition to this, the author thanks Renaud Alie, Alison Inglis, Oleksiy Klurman, Jiajun Mai, Sacha Mangerel, Kaisa Matom{\"a}ki, Andrei Shubin, Joni Ter{\"a}v{\"a}inen, and Peter Zenz for many helpful discussions and comments. Finally, some of this work was conducted while the author was visiting CalTech; he is grateful for their hospitality.

\bibliographystyle{alpha}
\bibliography{main}

\textsc{Department of Mathematics and Statistics, McGill University, 805 Sherbrooke St. W., Montreal, QC H3A 2K6, Canada}

\textit{Email address:} \href{mailto:iakov.chinis@gmail.com}{{\texttt{iakov.chinis@gmail.com}}}

\end{document}